\documentclass{llncs}
\usepackage{amssymb,amsmath}
\usepackage{algorithm,algorithmic}
\usepackage{epstopdf}
\usepackage{graphicx}
\usepackage{graphicx}
\usepackage{epstopdf}
\usepackage{caption}
\usepackage{amssymb}

\title{{\bf Spanning Tree Auxiliary Graphs}}
\institute{Dhirubhai Ambani Institute of Information  Communication Technology, Gandhinagar, India\\ \email{\{abhishekgarg2009@gmail.com, jadeja\_mahipal@daiict.ac.in, rahul\_muthu@daiict.ac.in\}}}
\author{Abhishek Garg, Mahipal Jadeja, Rahul Muthu}
\date{}
\begin{document}
\maketitle
\begin{quote}
{\bf Abstract} \small In this paper, we define a class of auxiliary graphs associated with simple undirected graphs. This class of auxiliary graphs is based on the set of spanning trees of the original graph and the edges constituting those spanning trees. We call our family of auxiliary graphs {\bf spanning tree auxiliary graphs (STAGs)}. In general, a class of auxiliary graphs can be viewed as a function from the class of graphs to the class of graphs. 

Not all graphs are STAGs of some graph. We prove some key structural properties of STAGs of simple undirected graphs and use them to develop an efficient algorithm to recognise if a graph is a STAG, and if so, compute the inverse (or original graph) from the STAG.

We focus on STAGs of 2-connected graphs. This is justified by another result we derive in this paper: The spanning tree auxiliary graph of a given graph is prime under the cartesian product operator on graphs, if and only if the graph is a 2-connected graph. A corollary to this result is that any spanning tree auxiliary graph has infinitely many preimages, by adding tree-like appendeges to a basic solution. In summary, we develop a result that establishes what one can call the core of the problem and then we develop an algorithm to reconstruct the inverse graph, for this core case.

Additionally we prove results relating parameters of a STAG to (not necessarily the same) parameters of the original graph.
\end{quote}
{\bf keywords}: Spanning trees, spanning tree auxiliary graph, blocks, 2-connected graphs, cartesian product

\section{Introduction}
The generic concept of auxiliary graphs is an important one in graph theory. In its most general form it refers to constructing graphs based on some rules applied to any given graph. In other words it is a function from the set of graphs to the set of graphs. The definition is usually based on some natural and important properties of graphs which get reflected in the auxiliary graph constructed. The computation of the function is easy in principle, even if the algorithm involved could be expensive in terms of computational complexity. What is usually less clear is the range of this function. It is rare for the range of these auxiliary functions to be the entire codomain (in this case the class of all graphs). Thus the challenging and important problems associated with these auxiliary graph families is to characterise mathematical properties of graphs which belong to the range, algorithms for deciding whether a graph belongs to the range or not and also algorithms for computing the inverse image of a given auxiliary graph if it is unique. If the preimage is not unique, then one interesting challenge is to decide what constitutes a minimal/canonical solution and also ways to generate the entire set of solutions by extending the basic solutions.

A well known example of auxiliary graphs is the class of line graphs. The characterisation as well as algorithms for recognising this class of graphs and computing their inverse images has been established in a wide variety of research articles \cite{behzad1970characterization} \cite{lehot1974optimal} \cite{roussopoulos1973max}.

Work on characterising and algorithmically recognising important well-defined classes of graphs has occupied a central place in the field of graph theory. This goes beyond just auxiliary graphs and examples include planar graphs, graphs that are prime under the cartesian product operator, interval graphs, perfect graphs and bipartite graphs. 

In this paper we study a class of auxiliary graphs where the vertices of the auxiliary graph represent the spanning trees of a given graph. There is an edge connecting two vertices of the auxiliary graph precisely when  the symmetric difference of the edge sets of the corresponding spanning trees has exactly two edges. That is equivalent to saying that the two spanning trees have $(n-2)$ of their $(n-1)$ edges common. 

Diagrammatically, one can label the vertices of a simple graph and also its edges with distinct labels. Given such a labelling of a graph $G$, one can label the vertices of the spanning tree auxiliary graph $Aux(G)$, each with the list of $(n-1)$ edges of the spanning tree it represents. From the description above it should be clear that we put an edge between two vertices in the spanning tree auxiliary graph if and only if the labels of the two vertices share $(n-2)$ of their $(n-1)$ elements in common. See the figure below for a graph $G$ and its spanning tree auxiliary graph $Aux(G)$.

	\begin{figure}[htb]
		\centering
		\begin{tabular}{@{}cccc@{}}
			\includegraphics[width=.40\textwidth]{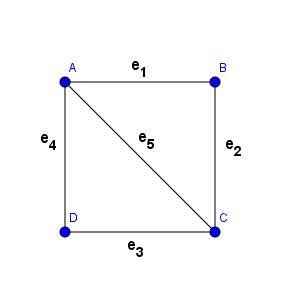} &
			\multicolumn{2}{c}{\includegraphics[width=.50\textwidth]{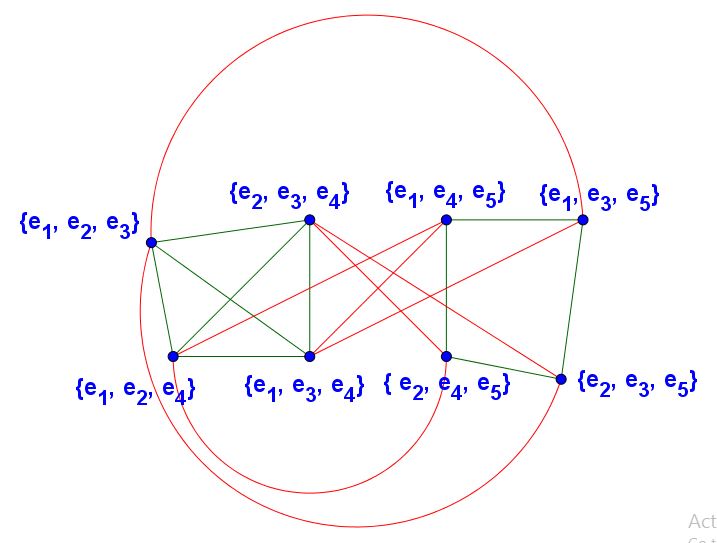}}
		\end{tabular}
		\caption{Construction of $Aux(G)$ from $G$}
	\end{figure}

This way of migrating from one spanning tree of a graph to another has already been studied in various places and a similar notion forms the basis of the proofs of correctness of algorithms such as Prim's and Kruskal's \cite{kruskal1956shortest} for computing minimum spanning trees in weighted graphs. Counting or enumerating the spanning trees of graphs has been extensively studied in the literature \cite{chaiken1978matrix} \cite{kapoor1995algorithms} \cite{shioura1997optimal}. This underlines the importance of this class of graphs. Apart from applications of this class of graphs in various problems as described here, it is also challenging combinatorially and algorithmically to characterise  this family of graphs.

Here, we formalise the notion of spanning tree auxiliary graphs of graphs and characterise them in terms of their mathematical properties. We develop algorithms for recognising such graphs and computing an inverse solution. We provide a complete description of all graphs which constitute preimages on the basis of one basic preimage (the spanning tree auxiliary graph of graphs is not an injective function and each point in the range has infinitely many preimages). We derive relations between parameters of a given graph and (not necessarily the same) parameters of the corresponding auxiliary graph. 

Throughout, we assume the original graph for which we are considering spanning tree auxiliary graphs is a simple undirected graph. We deal only with connected graphs because the set of disconnected graphs have no spanning trees and hence the corresponding spanning tree auxiliary graphs are trivial, with zero vertices.

The rest of this paper is organised as follows. In Section~\ref{secdef} we present definitions and concepts used. We also present therein some of the simpler results we derive, and use them to obtain our major results in other sections. We have done this to improve readability.  In Section~\ref{seccliq}, we provide a classification of all maximal cliques that occur in the class of spanning tree auxiliary graphs. Section~\ref{seccartprod} discusses the role of prime graphs under the cartesian product operator as the building blocks of all spanning tree auxiliary graphs. The spanning tree auxiliary graphs of all 2-connected graphs are shown to belong to the family of prime graphs under the cartesian product operator. We also establish in the same section, that for non-2-connected graphs, the spanning tree auxiliary graph is the cartesian product of the spanning tree auxiliary graph of each of its blocks (assuming the graph is connected). In Section~\ref{secalgo} we provide an algorithm that recognises a graph that is a spanning tree auxiliary graph of a simple graph and computes a basic preimage, from which all preimages may be generated. We derive results on a few elementary standard graph parameters for the family of spanning tree auxiliary graphs in Section~\ref{secpar}. We summarise the results and possible future directions of research in Section~\ref{secconcl}.

\section{Definitions}\label{secdef}
In this section, we present some elementary definitions, and some well know results that we find useful. To improve readability, we also derive some of our simpler results here, and refer to them in the other sections of the paper where they are used to obtain more substantial results. In order to organise the material better we have divided this section into subsections which group together definitions and results which are similar.

\subsection{trees, spanning trees and spanning tree auxiliary graphs}
\begin{definition}
	A tree is a simple undirected connected acyclic graph.
\end{definition} 
\begin{definition}
	A connected unicyclic graph is any graph obtained by augumenting a tree with an edge between a non-adjacent pair of vertices.
\end{definition}
\begin{flushleft}
	{\bf \underline{Unit Transformations}}\\
\end{flushleft}
{\bf \underline{Type $I$}}\\
Given any spanning tree $T$ of a simple connected undirected graph $G$, adding any edge $e\in E(G)\setminus E(T)$ results in a unicyclic graph $U$. Since, $G$ is a simple graph, the unique cycle in $U$   must necessarily be of length at least 3. Suppose it is of length $k$, then there are exactly $(k-1)$ non-cut edges in $U$ different from $e$. Deleting any one of them results in a spanning tree $T'$ of $G$, different from $T$. We call this process of adding an edge to a spanning tree of a connected graph and deleting {\bf some other} edge from the unique cycle thus introduced, a {\bf unit transformation of type $1$}. Any two spanning trees can be constructed from one another by a series of unit transformations of type 1, and in fact in at most $(n-1)$ unit transformations where $n$ is the number of vertices of $G$. 
\begin{flushleft}
	{\bf \underline{Type $II$}}\\
\end{flushleft}
Given any spanning tree $T$ of a graph $G$, deleting any edge $e$ from $T$ results in a spanning forest of $G$ consisting of exactly two trees $T_1$ and $T_2$. Adding any edge of the original graph different from $e$ and linking a vertex of $T_1$ to a vertex of $T_2$ results in a spanning tree $T'$ different from $T$. The number of such edges is equal to the number of edges in $G$ between the vertex partition defined by the vertices of $T_1$ and $T_2$. We call this process of deleting an edge of a spanning tree and relinking the two resulting subtrees by a {\bf different} edge a {\bf unit transformation of type $2$}. Any two spanning trees can also be constructed from one another by a series of unit transformations of type 2, and in fact in at most $(n-1)$ unit transformations where $n$ is the number of vertices of $G$. 

It should also be evident that any two spanning trees can be constructed from one another by a mixed series of type 1 and type 2 unit transformations, again requiring no more than $(n-1)$ steps in the most efficient way.

\begin{definition} 
	Given a simple graph $G$, we define its {\bf spanning tree auxiliary graph $Aux(G)$} as the graph which has a vertex corresponding to each spanning tree of $G$, and two vertices of $Aux(G)$ are adjacent if and only if the corresponding spanning trees in $G$ can be obtained by a single unit transformation. 
\end{definition}
In this paper the main goal is to characterise the set of all graphs which are $Aux(G)$ for some simple graph $G$, and also design an algorithm to reconstruct $G$ from a valid instance of $Aux(G)$. As a subsidiary goal, we also dervie results on various standard graph parameters for the class of STAGs.
\subsection{A specific algorithm for constructing spanning trees}
In this section, we present a new algorithm for computing a spanning tree of a graph. This algorithm is useful to us, not so much for the final answer it gives, as the sequence of steps it follows. Specifically, we are interested in the last step in the algorithm's execution.

The algorithm is a kind of reverse of Kruskal's famous algorithm for computing the minimum weight spanning tree of a graph. The other point of difference, apart from it being a reverse of Kruskal's algorithm is the fact that we are going to apply it to unweighted graphs. Unweighted graphs are really just weighted graphs, where each edge has weight 1. Thus, in this special case, every spanning tree is a minimum weight spanning tree. 

We add, that our algorithm, with its main goal of establishing a fact about its last step of execution is applicable {\bf only} for 2-connected graphs (this concept is defined in the next subsection, for readers not familiar with the concept or its properties). The goal of this algorithm is to establish, {\bf algorithmically}, the following result. 

\begin{theorem}\label{revkrustwoedges}
	For any spanning tree $T$ of a 2-connected graph $G$, and any two edges $e_1$ and $e_2$ on the spanning tree $T$, there is an edge $e$ of $G$ that is not an edge of $T$ such that the connected, spanning unicyclic graph $U=T+e$ contains both $e_1$ and $e_2$ on its unique cycle. 
\end{theorem}
\begin{proof}
	Kruskal's algorithm begins with $n$ isolated vertices and sorts the graph's edges in ascending (or non-descending) order of edge weights. It then considers the edges the edges one by one in this order and includes any edge that does not form a cycle with the earlier included edges, and rejects any edge that forms cycles with earlier included edges. 
	
	We propose a reverse of this idea. That is, we begin with the whole graph. We sort the edges of the graph in descending (or non-ascending) order. We then consider the edges in this order and reject any edge that lies on a cycle with edges not already rejected. It is easy to argue, in much the same way as Kruskal's algorithm, that this algorithm also finds a minimum weight spanning tree. 
	
	We tweak this algorithm, recognising the fact that we are dealing with unweighted graphs, in order to prove the theorem. We make a list of {\bf all cycles} in the graph, and order these cycles such that all cycles that contain {\bf both $e_1$ and $e_2$} are listed after all other cycles. Our proposed Kruskal reversed algorithm essentially destroys cycles in the graph until we are left with a spanning tree. All cycles must be destroyed, and the way the algorithm destroys any cycle is by removing at least one edge on that cycle. By keeping all cycles involving {\bf both $e_1$ and $e_2$} at the end, we are ensuring that the last cycle destroyed by our algorithm is one containing both $e_1$ and $e_2$. More specifically, the last edge deleted by our algorithm to give us our spanning tree is on a cycle containing both $e_1$ and $e_2$. This proves the theorem. 
\end{proof}

\subsection{2-Connected Graphs}
\begin{definition}
	A graph $G$ is $2${\bf-connected} if it cannot be disconnected by deleting fewer than two vertices. In particular, the graph itself must be connected, because otherwise, it is rendered disconnected by removing zero vertices, which is fewer than two.
\end{definition}
By definition, the graph $K_2$ is $2$-connected. Cycles $C_n$ with $n\ge3$ are also 2-connected. All other $2$-connected graphs are constructible by the process of addition of ears, by a result due to Whitney, that we present below.

\begin{definition}
	A {\bf block} in a graph is defined as any maximal $2$-connected subgraph of the graph.  
\end{definition}

It is an elementary result that any two blocks in a graph can share at most one common vertex. A useful auxiliary graph to study the block structure of a connected graph is the standard {\bf block-cutpoint tree} \cite{harary1966block} of a graph. The block-cutpoint tree  of a graph is computed by a standard algorithm which is an adaptation of depth first search ({\sc dfs}).

We would like to state at the very outset that there are infinitely many  graphs which all map to the same $Aux$ graph, and hence we need to develop a notion of a canonical/minimal preimage.  
\begin{definition}
	A {\bf minimal preimage} of a spanning tree auxiliary graph is a connected graph none of whose blocks is $K_2$. The blocks in such a listing maybe linked together in any form allowed by the standard block-cutpoint tree concept.
\end{definition}
The motivation behind the above definition is that the only changes to a graph that do not alter the spanning tree auxiliary graph are addition of blocks which are all $K_2$.

We now define the notion of {\bf ear addition} as used by us. The concept is not a new one, but our definition is slightly different and hence we present it here.
\begin{definition}
	We define an {\bf ear addition} as an extension of a graph by adding a path through zero or more new vertices with two distinct existing vertices of the graph as the endpoints of the path.
\end{definition}
If the endpoints of the path are already adjacent then the ear must contain at least one intermediate vertex since we consider only simple graphs. An {\bf ear decomposition} of a graph is the reconstruction of the graph from scratch by first drawing one of its cycles and then repeatedly  adding an ear.

We now state Whitney's Theorem \cite{whitney2009non} on 2-connected graphs..
\begin{theorem}[Whitney's Theorem]
	A graph is $2$-connected (apart from $K_2$) if and only if  it can be obtained by starting with a cycle and performing zero or more operations of ear addition.
\end{theorem}

\subsection{Cartesian Product}
Here, we present the definition of the well known operation of cartesian product of graphs, and some properties of this product that are useful for our work.
\begin{definition}
	Given two graphs $G_1=(V_1,E_1)$ and $G_2=(V_2,E_2)$, the cartesian product $H=G_1\Box G_2$ has vertex set $V=V_1\times V_2$ where $\times$ represents the cartesian product of the two vertex sets and an edge connects $(u_1,u_2)$ to $(v_1,v_2)$ if and only if $u_1=v_1$ and $(u_2,v_2)\in E_2$ or $(u_1,v_1)\in E_1$ and $u_2=v_2$.
\end{definition}

This operator defined for two graphs can be extended iteratively to any number of graphs. The operation is commutative and associative in the sense that the graphs obtained by commuting or bracketing a series of graphs in any order gives rise to the same product graph upto isomorphism. 

The graph $K_1$ serves as the identity for the cartesian product operator on graphs. It is well known that the nontrivial factors of a graph under the cartesian product operator are unique upto reordering.

\begin{definition}
	A graph which has only one nontrivial factor under the cartesian product operator is called prime. 
\end{definition}

A graph obtained as the cartesian product of $k$ nontrivial factors \cite{imrich2007recognizing} \cite{imrich1994factoring} is a graph of dimension $k$ under the cartesian product operation. Each vertex in the product graph involving $k$ nontrivial factors is a $k$ dimensional vector where the $i^{th}$ coordinate is a vertex from the $i^{th}$ factor in the product. 

\begin{lemma}
	Let $(v_1,\ldots,v_k)$ be a vertex in $G=G_1\Box\cdots\Box G_k$. Then \[d_G(v_1,\ldots,v_k)=d_{G_1}(v_1)+\cdots+d_{G_k}9v_k)\]
\end{lemma} 

\begin{lemma} \label{prime}
	Since nontrivial factors involve a minimum degree of at least 1, the presence of every edge of a vertex in the same factor implies the graph is prime.
\end{lemma}

The dimension of each vertex is identical and is the same as the dimension of the graph under cartesian product. Thus, in order to establish that a graph is prime under the cartesian product operator, it is enough to establish that all edges incident to some vertex belong to the same factor. 

\begin{lemma}
	All edges of a clique of size three or more in a cartesian product of graphs must all come from the same factor.
\end{lemma}
\begin{proof}
	This is easy to see, by looking at the definition of the cartesian product of graphs.
\end{proof}

\section{Classification of Maximal cliques in $Aux(G)$ in terms of structures in $G$}\label{seccliq}
Here we describe cliques on three or more vertices in $Aux(G)$. Each spanning tree of a graph $G$ has exactly $(n-1)$ edges where $n$ is the number of vertices in $G$. Consider a clique of size three in $Aux(G)$. This clique represents three spanning trees of $G$ each pair among which there is exactly $(n-2)$ common edges. There are two possibilities for the common intersection of the edge sets of all three spanning trees. Either it is $(n-3)$ or it is $(n-2)$. If we consider any fourth vertex to augment the three clique to a four clique, then in the first case, the common intersecton of the edge sets of the four spanning trees will go down to $(n-4)$, while in the second case it will remain $(n-2)$. The same logic extends to larger cliques. If it is a clique of type 1, then the common intersection decreases for each added vertex, while if it is of type 2, the common intersection remains $(n-2)$. Structurally cliques of type $1$ arise from cycles in $G$ and cliques of type $2$ arise from minimal edge cuts in $G$. 

\begin{lemma} \label{maxclique}
	Every clique of size $3$ or more in $Aux(G)$  uniquely extends to a maximal clique. Therefore these cliques can be computed in polynomial time.
\end{lemma}

The neighbourhood of each vertex in $Aux(G)$ can be partitioned into maximal cliques in these two different ways. In the first case the number of cliques in the partition is $m-n+1$ one corresponding to each edge of $G$ not in the spanning tree $T$. In the second case the number of cliques in the partition is $n-1$, one for each edge in the spanning tree $T$.

Each maximal clique in $Aux(G)$ is a direct and exclusive consequence of either a cycle in $G$ or a minimal edge-cut in $G$. The size of the cliques resulting in these two cases are the length of the cycle and the number of edges in the edge-cut respectively. To summarise:

{\bf Due to cycles:}\\ Take any cycle $C$ of length $k$ in $G$. Consider a spanning tree $T$ which uses some $(k-1)$ of the edges of this cycle. Let $F$ be the forest resulting by deletion of these $(k-1)$ edges from $T$. Clearly appending any path of $(k-1)$ edges of the cycle $C$ to the edges of $F$ result in a spanning tree of $G$ and differ from any other such tree in exactly one edge. Thus these are all pairwise adjacent and form a maximal clique of size $k$ in $Aux(G)$. 

{\bf Due to minimal edge-cuts:}\\ We assume $G$ is connected and let $\mathcal{E}$ constitute a minimal edge cut of $G$ containing $k$ edges. The deletion of the edges of $\mathcal{E}$ results in a two component graph. Take any fixed spanning forest of this two component graph containing spanning trees $T_1$ and $T_2$ of the two components respectively. Cross connecting $T_1$ and $T_2$ with any of the $k$ edges of $\mathcal{E}$ results in a spanning tree of $G$. Clearly each of these spanning trees differ from each other in exactly one edge. Thus, they constitute a maximal clique of size $k$ in $Aux(G)$.

There are two basic ways of creating a new spanning tree of a graph starting from a given spanning tree of the same graph. These are very similar to each other as single operations go but when we consider a series of these operations (or more precisely a large number of possibilities of completing the second phase of these operations) the difference between them becomes important and hence we consider both. These are presented in Section~\ref{secdef}.

\section{Minimal Preimage and multiple preimages}\label{seccartprod}
In this section, we establish that the STAG's of 2-connected graphs other than the trivial case $K_2$ are all prime under the cartesian product operator on graphs. We additionally demonstrate that the STAG of a general connected graph is the cartesian product of the STAGs of its individual blocks. As a consequence of this result, it follows that any STAG has infinitely many preimages, which maybe obtained by adding any number of $K_2$ blocks to a basic solution. Since the STAG of $K_2$ is the single vertex graph $K_1$, the identity for the cartesian product operation on graphs, this kind of change to the original graph doesn't result in any change in the STAG. These results are established over a series of theorems and lemmas.

\begin{theorem}
	The auxiliary graph of $G$ consisting of blocks $B_1,\ldots,B_k$ is the cartesian product of the individual auxiliary graphs. That is: \[Aux(G)=Aux(B_1)\Box \cdots \Box Aux(B_k)\]  
\end{theorem}

\begin{proof}
	Consider any graph $G$ and any spanning tree $T$ of $G$. Let $B$ be some block of $G$ and let $T[B]$ be the subgraph of $T$ induced by the vertices of $B$. Clearly $T[B]$ is a tree. If it were a forest with more than one component, it means there exists a path leaving the vertices of $B$ and coming back linking distinct vertices of $B$ via a path with vertices outside $B$. This contradicts the assumption that $B$ is a block. 
	
	Thus, the spanning trees of any connected graph can all be be obtained by patching together in any way individual spanning trees of each block of $G$. In fact any spanning tree of $G$ can be obtained by this procedure and  and any tree resulting from this patching together of spanning trees of blocks is also a spanning tree of $G$. 
	
	It follows from the above that any spanning tree of a graph can be viewed as an (ordered) list of spanning trees of its individual blocks. Different spanning trees of the graph can be obtained by starting with some spanning tree and then varying independently the spanning trees of each block. In other words the set of all spanning trees of the graph can be obtained as vectors of dimension $k$ where $k$ is the number of blocks of $G$. 
	
	We may also recall that in $Aux(G)$, two vertices (representing two distinct spanning trees in $G$) are adjacent if and only if they can be obtained from each other via a unit transformation. Also the edges involved in this unit transformation must both come from the same block of $G$ since they form a part of a cycle in $G$ and there can be no cycle crossing more than one block. Thus we can say that the two "adjacent" spanning trees agree in their restriction to all blocks except one, and on the one where they disagree, they differ by a unit transformation. If we were to treat these spanning trees as $k$ dimensional vectors one for each block of $G$, then the STAG is the cartesian product of the individual STAGs of each block.
\end{proof}

Here we describe the properties which make two or more graphs map to the same $Aux$ graph. 

\begin{lemma}
	If $G'$ is obtained from a graph $G$ by iteratively appending new blocks to $G$ each of which is a $K_2$, then it results in no change, and $Aux(G')$ from $Aux(G)$.  
\end{lemma}

This is because $Aux(K_2)=K_1$ and $Aux(G')=Aux(G)\Box K_1=Aux(G)$.

\begin{definition}[common cycle membership relation]
	For any undirected graph $G$ let us define a binary relation over the edge set $E$. We say two edges are related if they lie on a common cycle.
\end{definition}

\begin{lemma}
	The relation defined above is an equivalence relation when $G$ is a bridgeless (cut-edge free) graph. The equivalence classes in this case are the edges of any block of the graph.
\end{lemma}
\begin{proof}
	Since we assume the graph $G$ is bridgeless, every edge lies on a cycle, and hence every edge trivially lies on a cycle containing itself. Thus the relation is reflexive.
	
	Clearly if $e_1$ lies on a common cycle with $e_2$ then $e_2$ also lies on the same common cycle with $e_1$. Thus the relation is symmetric.
	
	The only nontrivial property is transitivity. Consider a bridgeless graph containing three distinct edges $e_1$, $e_2$ and $e_3$. Suppose $e_1$ and $e_2$ both lie on a common cycle $C_1$. Suppose that $e_2$ and $e_3$ lie on a common cycle $C_2$. We will show that there is a cycle $C_3$ that contains both $e_1$ and $e_3$. 
	
	We start on edge $e_3$ on cycle $C_2$ and move in both directions along the cycle, until we encounter the first intersections with cycle $C_1$ in each direction. Such intersecting points must exist and be distinct, since the two cycles share the edge $e_2$. Let these intersecting points be vertices $x$ and $y$. There is an $x$ to $y$ path on $C_1$ containing the edge $e_1$. This is because for any two distinct vertices on a cycle, and an edge on the same cycle, there is a path on the cycle between the two vertices passing through that edge. These fragments (paths) from the cycles $C_2$ and $C_1$ that meet at $x$ and $y$ form a cycle containing both $e_1$ and $e_3$. 
\end{proof}
\begin{lemma} \label{onec}
	In $2$-connected graphs, every pair of edges has at least one cycle containing both.
\end{lemma}
\begin{proof}
	Proof is by induction on the number of ears, as per Whitney's decomposition. The base case is a cycle. Clearly in this case each pair of edges has a common cycle on which they lie. Let us consider the induction step when we add a new ear. The newly added ear with endpoints $x$ and $y$ forms a cycle together with some $x$ to $y$ path in the graph before the ear was added. By induction hypothesis for every pair of edges, prior to adding this ear, there is a cycle containing both.  By the equivalence relation proved above, our claim is established.
\end{proof}

\begin{theorem}\label{2con}
	Let $G$ be a 2-connected graph, different from $K_2$, and let $H=Aux(G)$. Then $H$ is a prime graph under the cartesian product operation.
\end{theorem}
Here we focus on an arbitrary vertex $x$ in $Aux(G)$ and argue that all the edges incident to $x$ in $Aux(G)$ come from the same factor. This will imply that $Aux(G)$ is prime under the cartesian product operator. We will, of course, have to use the fact that $G$ is $2$-connected in the course of our proof. 
Consider the spanning tree $T_x$ of $G$ corresponding to the vertex $x$ in $Aux(G)$. The edges incident to $x$ in $Aux(G)$ connect it to its neighbours. Hence, these correspond to spanning trees of $G$ obtained from $T_x$ by a single unit transformation. Consider an edge $e_1$ in $T_x$. Deleting $e_1$ from $T_x$ results in a spanning forest of $G$ with exactly two trees $T_1$ and $T_2$. Since $G$ is $2$-connected, there is at least one edge in $G$ apart from $e_1$ linking the vertices of $T_1$ to the vertices of $T_2$. The spanning trees obtained by reconnecting $T_1$ and $T_2$ using any of these edges all constitute a minimal edge cut clique. Hence, all these edges connecting $x$ to this set of neighbours come from the same factor of a cartesian product.

Similarly, all neighbours of $x$ obtained by deleting an edge $e_2$ and applying a unit transformation of type $II$ also form a clique. Hence, all edges to this group of neighbours of $x$ also come from the same factor of a cartesian product. 

Now we just need to establish that at least one edge from each of these groups together come from the same factor of a cartesian product. Since $G$ is 2-connected, by applying Theorem~\ref{revkrustwoedges} we know there is an edge $e$ in $G$ and not in $T_x$, such that $U_{x,e}=T_x+e$ is a connected, spanning unicyclic graph with its unique cycle containing both $e_1$ and $e_2$. The trees obtained by adding $e$ and deleting in turn $e_1$ and $e_2$ result in two neighbours of $x$ that belong to the cycle clique associated with the addition of edge $e$ to $T_x$. Thus the edges between $x$ and these two neighbours come from the same factor. However these two edges lie in respetively the minimal edge cuts associated with deletion of $e_1$ and $e_2$. Thus all these edges (and by extension all edges incident to $x$) come from the same factor of a cartesian product. This proves that the auxiliary graph is prime under the cartesian product operation.

\begin{figure}[htb]
	\centering
	\includegraphics{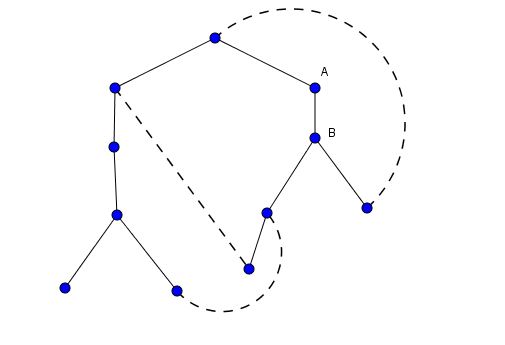}
	\caption{ Idea of  Type $I$ and Type $II$ cliques for edge (A.B) }
\end{figure}

Thus $Aux(G)$ is prime if $G$ is 2-connected and has no cut-edges.  

\begin{lemma}
	$Aux(T)=K_1$ for any tree $T$. 
\end{lemma}
This follows because a tree has exactly one spanning tree. Note that the spanning tree auxiliary graph of a tree- the complete graph on one vertex, $K_1$ - is also the identity for the cartesian product operator. Thus, from the previous two theorems, we conclude, that appending any number of blocks which are trees to a given graph, does not alter the spanning tree auxiliary graph of that graph. Thus minimal preimages contain no cut-edges.
\section{Characterisation of $Aux$}\label{secalgo}
Here we describe some properties of spanning tree auxiliary graphs and use them to develop an algorithm to compute the inverse. The first subsection deals with reconstruction of inverse, for the case when the inverse is a 2-connected graph. The second subsection extends the algorithm to graphs that are not 2-connected.

\subsection{2-connected graphs}
As described in Section~\ref{seccliq}, we can partition the neighbourhood of any vertex in a spanning tree auxiliary graph of a 2-connected graph, into cliques in two ways. Such a partition gives us two sets of two simultaneous equations in the variables $n$ (number of vertices of $G$) and $m$ (number of edges of $G$). Solving these equations, we can determine the number of vertices and edges of $G$. This calculation must be consistent across all vertex neighbourhoods of any candidate STAG. In addition to the number of vertices, we get the sizes and structure of all the minimal edge cut and cycle type cliques in the neighbourhood of a vertex of $Aux(G)$. 

Assuming the information is consistent across all vertex neighbourhoods, our algorithm proceeds as follows. We compute $n$ and $m$. We use the partitions of the neighbourhood of a fixed vertex $x$ in the STAG $Aux(G)$ into maximal cliques in two ways (minimal edge cuts and cycles). We label the minimal edge cut cliques as $\mathcal{E}_1,\ldots,\mathcal{E}_{n-1}$ and the cycle cliques as $\mathcal{C}_1,\ldots,\mathcal{C}_{m-n+1}$. We now give a secondary label to each minimal edge cut clique, as the set of cycle cliques it intersects. Thus, we have a labelling of each of the minimal edge cut cliques as a subset of the set of cycle cliques. This list of $n-1$ labelled minimal edge cut cliques can be treated as the $n-1$ edges of the spanning tree. These edges are labelled by cycle cliques, which can also be viewed as the fundamental cycles passing through the particular edge of the spanning tree.

\begin{lemma}\label{Spanningtreelabels}
	The following are properties of the decomposition of the neighbourhood of a vertex in the STAG of a 2-connected graph.
	\begin{itemize}
		\item The set of all spanning tree edges containing a particular fundamental cycle in its label forms a path.
		\item No two fundamental cycles are used in the labels of an identical set of tree edges.
		\item Any two incident tree edges have at least one fundamental cycle common to their labels.
	\end{itemize} 
\end{lemma}

Our reconstruction algorithm thus consists of two major phases. The first is to obtain a fundamental-cycle labelled spanning tree. The second is to add the non-tree edges to the spanning tree to reconstruct the original graph. The second phase uses the labelling of the spanning tree edges by cycle cliques.

\begin{flushleft}
	{\bf \underline{Phase 1: Constructing the spanning tree:}}
\end{flushleft}
\begin{enumerate}
	\item Use the two partitions of the neighbourhood of the vertex into maximal cliqiues to obtain a labelling of all the minimal edge cut cliques by subsets of cycle cliques. 
	\item Treat the labels obtained in Step 1, as the $n-1$ edges of the spanning tree labelled by the fundamental cycles they are involved in.
	\item Consider an arbitrary fundamental cycle and form a path using all the edges containing it in their labels. Now rearrange the edges of this path until all edges containing any particular fundamental cycle in their labels are subpaths of the path. 
	\item Repeat Step 3, to include the remaining edges and obtain a tree. 
\end{enumerate}
\begin{flushleft}
	{\bf \underline{Phase 2: Adding the non-tree edges to the spanning tree}}
\end{flushleft}
\begin{enumerate}
	\item For each fundamental cycle locate the endpoints of the path in the spanning tree, consisting of all edges using this cycle in its label.
	\item Add an edge between the two endpoints of each path obtained in step 1, to complete reconstruction of the graph.
\end{enumerate}

\subsection{General Graphs}
Given an arbitrary graph we first compute its prime factors under the cartesian product operation using well known algorithms  \cite{factoralgpap}. The result of the previosu subsection applies only to $2$-connected instances of $G$ (and consequently $Aux(G)$ is prime under cartesian product). For graphs which are not $2$-connected the algorithm uses Theorem~\ref{2con} to reduce into several subproblems and then apply the above result. 

We can then put together the blocks obtained as solutions for each factor in the cartesian product, using the block-cut-point tree model to obtain various inverse images.

We can extend the basic solutions by adding any number of tree-like appendages as we wish.

\subsection{Analysis}
We give here a brief informal analysis of running time of the algorithm provided by us. 
\begin{itemize}
	\item Prime factors of the input graph can be computed in polynomial time  \cite{feigenbaum1985polynomial}.
	\item From Observation~\ref{maxclique}, it is possible to compute all the maximal cliques in polynomial time (assuming the graph is the spanning tree auxiliary graph of some graph). This can be done in $O(n^4)$ time because there are $O(n^3)$ triangles and each extends greedily to unique  maximal clique.
	\item If the decomposition of the previous step is consistent across all the vertices then $n(G)$ and $m(G)$ can be computed in polynomial time.
	\item The two phase reconstruction algorithm of spanning tree followed by non-tree edges can be done in polynomial time.
\end{itemize}

\section{Parameters}\label{secpar}
In this section we give some elementary results on some standard graph parameters of spanning tree auxiliary graphs of graphs.
\begin{lemma}
	{\bf Max. degree:}\[\Delta(Aux(G))\le (n(G)-1)*(m(G)-n(G)+1)\]
\end{lemma}
\begin{proof}
	For a vertex of $Aux(G)$ there is an associated , the number of edges of $G$ not belonging to it is $m(g)-n(G)+1$. For each of those edges, adding them to the tree results in a cycle. The length of this cycle is at most $n(G)$, and thus the number of edges on the cycle, different from the one that was added is at most $n(G)-1$. Removing any of these edges generates a new spanning tree of $G$ and thus a neighbour of the vertex considered in $Aux(G)$. Combining these observations gives the upper bound on the maximum degree. 
\end{proof}
\begin{lemma}
	{\bf Min Degree:}
	\[\delta(Aux(G))\ge2*((m(G)-n(G)+1)\]
\end{lemma}
\begin{proof}
	The proof is almost identical to the previous lemma, the only difference being that we use the lower bound on the length of the cycle created, rather than the upper bound. The lower bound is 3, since we are dealing with simple graphs. The rest of the arguments are identical.
\end{proof}
\begin{lemma}
	The diameter $diam(Aux(G))\le n(G)-1$.
\end{lemma}
\begin{proof}
	The minimum number of operations to transform one spanning tree to the other is the size of the set difference of the edge sets of the two spanning trees. This can never be more than the number of edges in the tree, this bound being achieved in case of edge disjoint spanning trees. Thus the result follows.
\end{proof}

\begin{lemma}
	The {\bf Clique Number:}
	\[\omega(Aux(G))=Max\{{Circumference(G)},|{maximum\_ minimal \_ edge cut(G)}|\}\]
\end{lemma}
\begin{proof}
	Every maximal clique in $Aux(G)$ corresponds to either a cycle in $G$ or a minimal edge cut in $G$ as explained along with the definitions of the two types of unit transformations. Thus the maximum clique size in $Aux(G)$, which is necessarily a maximal clique is a largest among these. Thus the result follows.
\end{proof}

\section{Conclusions}\label{secconcl}
We have looked at the important class of spanning tree auxiliary graphs and given a mathematical characterisation of such graphs. We have also developed efficient algorithms to obtain the original graph given a spanning tree auxiliary graph. A possible direction of future research, is to make minimum changes to a graph that is not a STAG, to obtain a new graph that is a STAG.

\bibliographystyle{plain} 
\bibliography{spaux}
\end{document}